\documentclass{amsart}
\usepackage{amsmath,amssymb,epic,graphicx,mathrsfs,enumerate,xcolor,amsthm,hyperref,bbold}
\usepackage{graphicx} 

\newcommand{\SL}{{\operatorname{SL}}}
\newcommand{\PGL}{{\operatorname{PGL}}}

\newcommand{\PSL}{{\operatorname{PSL}}}

\newcommand{\PSU}{{\operatorname{PSU}}}
\newcommand{\Sp}{{\operatorname{Sp}}}
\newcommand{\PSp}{{\operatorname{PSp}}}
\newcommand{\Sz}{{\operatorname{Sz}}}
\newcommand{\OO}{{\operatorname{\mbox{P}\Omega}}}

\newtheorem{defi}{Definition}[section]
\newtheorem{teor}[defi]{Theorem}
\newtheorem{cor}[defi]{Corollary}
\newtheorem{prop}[defi]{Proposition}

\newtheorem{lemma}[defi]{Lemma}
\newtheorem{conj}[defi]{Conjecture}
\newtheorem{rem}[defi]{Remark}

\title[HS conjecture for simple groups]{The Herzog-Sch\"{o}nheim conjecture for simple and symmetric groups}

\author{Martino Garonzi}
\address{Martino Garonzi. University of Ferrara (Italy), Dipartimento di Matematica e Informatica.
ORCID: https://orcid.org/0000-0003-0041-3131}
\email{martino.garonzi@unife.it}

\author{Leo Margolis}
\address{Universidad Aut\'onoma de Madrid, Departamento de Matem\'aticas,  C/ Francisco Tom\'as y Valiente 7, Facultad de Ciencias, m\'odulo 17 and 
Instituto de Ciencias Matem\'aticas, C/ Nicol\'as Cabrera 13, 28049 Madrid, Spain}
\email{leo.margolis@uam.es, leo.margolis@icmat.es}



\thanks{Martino Garonzi was supported by FIRD 2025 (unife). Leo Margolis was supported by the Spanish Ministry of Science and Innovation, through the Ram\'on y Cajal grant program. We express our gratitude to the Universities of Brasilia and Padova for their hospitality.}

\begin{document}

\begin{abstract}
The Herzog-Sch\"{o}nheim conjecture states that if $H_1$, $\ldots$, $H_k$ are subgroups of a group $G$ and $x_1,\ldots,x_k$ are elements of $G$ such that $H_1x_1$, $\ldots$, $H_kx_k$ is a partition of $G$ into cosets, then two of these subgroups must have the same index. We prove this conjecture for simple groups and for symmetric groups.
\end{abstract}

\maketitle

\section{Introduction}

Erd\H{o}s conjectured in the 1950's that in any non-trivial partition of the integers by arithmetic progressions the biggest common difference appearing comes up at least twice. I.e. if for some $k \geq 2$ we have $\mathbb{Z} = \bigcup_{i=1}^k (n_i \mathbb{Z} + d_i)$ where $n_1 \leq ... \leq n_k$ and $d_1$,...,$d_k$ are positive integers and $(n_i\mathbb{Z} + d_i) \cap (n_j\mathbb{Z} + d_j) = \emptyset$ for $i \neq j$, then $n_{k-1} = n_k$. This conjecture was proved independently by Davenport, Mirsky, Newman and Rado and opened a fruitful line of research about so called covering systems. See \cite{PorubskySchonheim} for more historic details and \cite{Ginosar} for an algebraic proof of Erd\H{o}s' conjecture.

Viewing the integers as an additive group and the arithmetic progressions as cosets inside this group, this result lead Herzog and Sch\"{o}nheim to put forward the following conjecture in 1974:

\begin{conj}[Herzog-Sch\"{o}nheim Conjecture \cite{HerzogSchonheim}]
Let $G$ be a group. Assume $H_1$,...,$H_k$ are proper subgroups of $G$, with $k \geq 2$, and $x_1$,...,$x_k$ elements of $G$ such that $G = \bigcup_{i=1}^k H_ix_i$ and $H_ix_i \cap H_jx_j = \emptyset$ for $i \neq j$. Then $[G:H_i] = [G:H_j]$ for some $i \neq j$.
\end{conj}

It is known that solving the Herzog-Sch\"onheim Conjecture for all groups is equivalent to solving it for finite groups and that even for an infinite group $G$ a finite  partition into cosets only involves subgroups of finite index \cite{KorecZnam}. The conjecture is known to hold for several classes of groups $G$, e.g. if $G$ possesses a Sylow tower (so, in particular for supersolvable groups) \cite{BFF}, if the order of $G$ is divisible by few primes \cite{GinosarSchnabel} or if $|G| \leq 1440$ \cite{MS}. Cf. also \cite{AS, Chouraqui19, Chouraqui21} for some recent papers dealing with the conjecture from other perspectives. In particular, \cite{Chouraqui19} considers free groups of finite rank. An easy observation is that, if a group $G$ satisfies the conjecture, then any quotient of $G$ does too.

However, until now the Herzog-Sch\"onheim Conjecture has rarely been considered for finite non-solvable groups. To our knowledge, only the groups $A_5$, $S_5$, $\Sz(8)$ and $\Sz(32)$ appeared explicitly in the literature \cite{GinosarSchnabel}, and the only other finite non-solvable groups for which the conjecture is known to hold are those of order less than 1440. 

We prove the following two results.

\begin{teor}\label{th:Symmetric}
The Herzog-Sch\"onheim Conjecture is true for symmetric groups.    
\end{teor}

\begin{teor}\label{th:Simple}
The Herzog-Sch\"onheim Conjecture is true for simple groups.    
\end{teor}


Our basic idea consists in using the fact that maximal subgroups of simple and symmetric groups are typically small, so that partitioning the group by cosets of subgroups of different order is in most cases impossible already for arithmetical reasons. Formally, let $G$ be a finite group and let $m_1, \ldots, m_k$ be the indices of the subgroups of $G$ (including $\{1\}$ and $G$), counted without multiplicity, so that $m_i \neq m_j$ for $i \neq j$. Define 
$$\mathcal{J}(G) := \sum_{i=1}^k \frac{1}{m_i}$$
and note that $\mathcal{J}(G) \geq 1$. 

If a finite group $G$ satisfies $\mathcal{J}(G) < 2$, then the Herzog-Sch\"onheim conjecture has a positive answer for $G$. Indeed, assume $G = \bigcup_{i=1}^k H_ix_i$ and $H_ix_i \cap H_jx_j = \emptyset$ for $i \neq j$ for some $H_1,\ldots,H_k \leq G$ and $x_1,\ldots,x_k \in G$. Then $|H_1| + \ldots + |H_k| = |G|$, so that 
\[\sum_{i=1}^k \frac{1}{[G:H_i]} =1. \]
Hence, if $[G:H_i] \neq [G:H_j]$ for $i\neq j$, one must have $\mathcal{J}(G) \geq 2$ as $1 = [G:G]$ is a summand in the definition of $\mathcal{J}(G)$ and $[G:H_i] \neq 1$ for all $i$.

This idea also underlies e.g. the proof of the conjecture for $A_5$ in \cite{GinosarSchnabel}. In fact, we are able to show that $\mathcal{J}(G) < 2$ holds for all finite simple groups and also $\mathcal{J}(S_n) < 2$ for $n \geq 7$. To prove this we rely on the Classification of Finite Simple Groups and the extensive literature on maximal subgroups of finite simple and symmetric groups. We also obtain an asymptotic behavior of the function for simple groups.

\begin{teor}\label{th:Asymptotic}
If $S$ denotes a simple group, then $\displaystyle \lim_{|S| \to \infty} \mathcal{J}(S) = 1$.
\end{teor}

Note that $|G| \cdot \mathcal{J}(G)$ is precisely the sum of the subgroup orders of $G$ (counted without multiplicity). For example, if $p$ is an odd prime number and $G$ is any group of order $2p$, then $\mathcal{J}(G)= \frac{1}{2p} \left( 1+2+p+2p \right) = \frac{3(p+1)}{2p}$. An important example is the following: if $C_n$ denotes the cyclic group of order $n$, then $\mathcal{J}(C_n)=\sigma(n)/n$ where $\sigma(n)$ is the sum of the positive divisors of $n$. Moreover, for elementary abelian groups, we have $\mathcal{J}({C_p}^n) = \frac{(1/p)^{n+1}-1}{(1/p)-1}$.

Also note that $\mathcal{J}(G)$ is unbounded in general, for example if $p_1,p_2,\ldots$ are the prime numbers in increasing order and $G_n$ is the cyclic group of order $p_1 \ldots p_n$, then $\mathcal{J}(G_n)$ tends to infinity with $n$, because $\mathcal{J}(G) \geq \sum_{i=1}^n 1/p_i$ and the sum of the inverses of all the prime numbers is a divergent series. Hence, the structure of the underlying groups is really essential for our arguments.

The paper is structured as follows. In Section 2 we lay the groundwork for the other parts of the paper and verify that $\mathcal{J}(G) < 2$ holds for sporadic simple groups. In Section 3 we deal with alternating and symmetric groups, in Section 4 with simple groups of Lie type and finally in Section 5 we give the proofs of the main theorems which easily follow from the work done previously. We also include a remark on the limitations of our method.
To handle groups of small order and to estimate several functions we rely on computer calculations. For those we use GAP \cite{GAP} and Mathemtica \cite{Mathematica}. The code to verify our calculations is available in a publicly accessible GitHub repository \cite{GitRepGAP, GitRepMathematica}. 

\textbf{Thanks:} We express our gratitude to Colva Roney-Dougal, David Craven and Martin Liebeck for consultations on the maximal subgroups of simple groups of Lie type. In fact, the proof of Proposition~\ref{E8lbound} is due to Liebeck. We also thank the referee very many useful comments.

\section{Preliminaries and sporadic simple groups}

We will say that a finite group $G$ is \emph{HS}, if it satisfies the Herzog-Sch\"onheim conjecture.
In this section we develop the general tools we will use during the paper and apply them to show that the sporadic simple groups are HS.

\begin{lemma} \label{basics}
Let $G$ be a finite group and let $N \unlhd G$.
\begin{enumerate}
\item If $\mathcal{J}(G) < 2$, then $G$ is HS.
\item $\mathcal{J}(G) \leq \mathcal{J}(N) \mathcal{J}(G/N)$.
\end{enumerate}
\end{lemma}
\begin{proof}
(1) Assume $H_1$, $\ldots$, $H_k \leq G$, where $k \geq 2$ and $x_1,\ldots,x_k \in G$ are such that $G = \bigcup_{i=1}^k H_ix_i$ and $H_ix_i \cap H_jx_j = \emptyset$ for $i \neq j$. Set $n_i = [G:H_i]$, so $n_i \neq 1$ for all $i$. Assume that $n_i \neq n_j$ for all $i \neq j$, so that $G$ is not HS. Then, $1 = \sum_{i=1}^k 1/n_i \leq \mathcal{J}(G)-1 < 1$, a contradiction.

(2) If $H \leq G$, then $|H| = |HN/N| \cdot |H \cap N|$. This implies that every index of a subgroup of $G$ is the product between an index of a subgroup of $N$ and an index of a subgroup of $G/N$, so its inverse appears as a summand in $\mathcal{J}(N)\mathcal{J}(G/N)$. The claim follows.
\end{proof}

Let $\gamma$ be the Euler-Mascheroni constant, so that $1.78 < e^\gamma < 1.79$. We write $\ln$ for the natural logarithm. Consider the following function, which will be crucial in our work. It is defined on $\{1,2\} \cup [3,\infty)$.
\begin{equation} \label{defh}
\mathcal{B}(x) := \left\{ \begin{array}{ll} e^{\gamma} \ln\ln x + \frac{0.6483}{\ln\ln x} & \mbox{if } x \in [3,\infty), \\
1 & \mbox{if } x=1, \\
3/2 & \mbox{if } x=2.
\end{array} \right.
\end{equation}
For $x > 3$ we have $B'(x) = \frac{e^{\gamma} (\ln\ln x)^2-0.6483}{x \ln x (\ln \ln x)^2}$, so $B(x)$ is increasing when $(\ln \ln x)^2 \geq 0.6483/e^{\gamma}$, equivalently $x \geq \exp(\exp(\sqrt{0.6483/e^{\gamma}}))$, which is a number strictly between $6.22$ and $6.23$. So $B(x)$ is increasing for $x \geq 6.23$ \cite{GitRepMathematica}.

Let $\sigma(n) := \sum_{d \mid n} d = n \cdot \mathcal{J}(C_n)$. The bound $\sigma(n) \leq n \cdot \mathcal{B}(n)$ holds for all $n \in \mathbb{N}$. This is trivial for $n \in \{1,2\}$ and it follows from \cite[Th\'eor\`eme 2]{Robin84} otherwise.  We will use it several times without further mention. The reference \cite[Th\'eor\`eme 2]{Robin84} is also the reason for our definition of the function $\mathcal{B}$, as it provides an upper bound for $\sigma(n)/n$.

\begin{lemma} \label{f2}
Let $G$ be a finite group.
\begin{enumerate}
\item $\mathcal{J}(G) \leq \mathcal{B}(|G|)$.
\item Let $G$ be a finite group, let $\mathcal{M}$ be a set of representatives of the maximal subgroups of $G$ up to isomorphism and let $M_1,\ldots,M_k \in \mathcal{M}$ with $M_i \neq M_j$ for all $i \neq j$. Let $\mathcal{O} = \{|M|\ |\ M \in \mathcal{M} \setminus \{M_1,\ldots,M_k\}\}$. Then
\begin{equation}
\mathcal{J}(G) \leq 1+\sum_{i=1}^k \frac{\mathcal{J}(M_i)}{[G:M_i]} + \sum_{m \in \mathcal{O}} \frac{\mathcal{B}(m)}{|G|/m}.
\end{equation}
\end{enumerate}
\end{lemma}

\begin{proof}
(1) Let $n=|G|$. Then 
$$\mathcal{J}(G) \leq \sum_{d \mid n} \frac{1}{d} = \frac{1}{n} \sum_{d \mid n} \frac{n}{d} = \frac{\sigma(n)}{n} \leq \mathcal{B}(n).$$

(2) Let $h$ be an index of a subgroup $H$ of $G$.  If $H$ is isomorphic to a subgroup of one of the groups $M_1$,...,$M_k$, say to a subgroup of $M_i$, then $1/h$ appears as a summand in $\mathcal{J}(M_i)/[G:M_i]$. Otherwise, $1/h$ is a summand of $\sum_{d\mid m} d/|G| = \sigma(m)/|G|$ for some $m \in \mathcal{O}$. The claim follows from $\sigma(m) \leq m \mathcal{B}(m)$.
\end{proof}

Let $n$ be a natural number. Let $m_1 \leq \ldots \leq m_n$ be the $n$ smallest indices of maximal subgroups of $G$ up to isomorphism and $\ell$ the number of isomorphism classes of maximal subgroups of $G$. Let
\[\mathcal{B}_n(G) := (e^\gamma+1) \cdot \left( \sum_{i=1}^{n-1} \frac{\ln\ln(|G|/m_i)}{m_i} + \frac{(\ell-n+1)\ln\ln(|G|/m_n)}{m_n} \right)\]
In particular,
$$\mathcal{B}_1(G) = (\ell/m_1) \cdot (e^{\gamma}+1) \cdot \ln\ln(|G|/m_1).$$

\begin{lemma}\label{lem:GeneralSimplification}
Assume $n$ is a positive integer and $G$ is a finite group such that each maximal subgroup of $G$ has order at least $7$. With the above notation, if $\mathcal{B}_n(G) < 1$, then $\mathcal{J}(G) < 2$. 
\end{lemma}

\begin{proof}
Let $M$ be a maximal subgroup of $G$. If $|M| \geq 16$, then $\mathcal{J}(M) \leq \mathcal{B}(|M|) \leq (e^{\gamma}+1) \ln\ln |M|$ simply because $\ln\ln |M| \geq 1$, as this is equivalent to $|M| \geq e^e$ and $15 < e^e < 16$. If $7 \leq |M| \leq 15$, then $\mathcal{J}(M) \leq (e^{\gamma}+1) \ln \ln |M|$ by inspection \cite{GitRepGAP, GitRepMathematica}.  So, by Lemma~\ref{f2},
$$\mathcal{J}(G)-1 \leq \sum_M \frac{\mathcal{J}(M)}{[G:M]} \leq (e^{\gamma}+1) \sum_{M} \frac{\ln\ln |M|}{[G:M]} \leq \mathcal{B}_n(G) < 1,$$ 
where the sum is over representatives of maximal subgroups of $G$ up to isomorphism, the penultimate inequality follows from the definition of $\mathcal{B}_n$ and the last inequality holds by assumption. So, the result follows.
\end{proof}

To apply this lemma for simple groups, we will need the following observation.

\begin{lemma} \label{maxs3}
Let $G$ be a finite simple group with a maximal subgroup $M$ such that $|M| < 8$. Then $G \cong A_5$ and $M \cong S_3$.
\end{lemma}
\begin{proof}
By Herstein's theorem \cite[Theorem 5.53]{machi}, a finite nonsolvable group cannot have abelian maximal subgroups, so $M$ is nonabelian. It follows that $M \cong S_3$. Let $P$ be a Sylow $3$-subgroup of $G$ containing the Sylow $3$-subgroup $Q$ of $M$. Assume that $P \neq Q$. Since $P$ is a $3$-group, there exists $x \in P \setminus Q$ which normalizes $Q$. By the maximality of $M$, $Q$ is normal in $\langle M,x \rangle = G$, contradicting the fact that $G$ is simple. So $P=Q$, in other words, $P$ is contained in $M$. The normalizer $N_G(P)$ contains $M$ so it is equal to $M$ by maximality. It follows that $G$ does not have elements of order $6$. Using \cite{podufalov} we deduce that $G$ is one of the following: $\PSL(2,q)$, $\PSL(3,2^n)$, $\PSU(3,2^n)$, $\Sz(2^n)$. Inspection using \cite[Tables 8.1, 8.2, 8.3, 8.4, 8.5, 8.6, 8.16]{BHR} gives the result.
\end{proof}

\begin{cor}\label{cor:SimplificationSimpleGroups}
Let $G$ be a finite simple group and $n$ a natural number. If $\mathcal{B}_n(G) < 1$, then $ \mathcal{J}(G) < 2$.    
\end{cor}
\begin{proof}
We have $\mathcal{J}(A_5)=103/60 < 2$ by \cite{GitRepGAP}. So, by Lemma \ref{maxs3}, we may assume that every maximal subgroup of $G$ has order at least $8$. The result follows from Lemma~\ref{lem:GeneralSimplification}.
\end{proof}

We will also use the following version of Stirling's inequalities (see \cite{Robbins}).

\begin{lemma}\label{Stirling}
For every positive integer $m$ we have
$$\sqrt{2 \pi m} (m/e)^m e^{\frac{1}{12m+1}} < m! < \sqrt{2 \pi m} (m/e)^m e^{\frac{1}{12m}}.$$
\end{lemma}

These basic ideas and the information available in the ATLAS \cite{ATLAS} is already enough to confirm the conjecture for sporadic simple groups.

\begin{prop}\label{Sporadics}
    If $G$ is a sporadic simple group or the Tits group ${}^2F_4(2)'$, then $\mathcal{J}(G) < 2$. In particular, $G$ is HS.
\end{prop}
\begin{proof}
    For all sporadic simple groups except the Mathieu group of degree 11 we can use Corollary~\ref{cor:SimplificationSimpleGroups}, where the indices of maximal subgroups as well as the number of isomorphism types of maximal subgroups of $G$ are all contained in \cite{ATLAS} as well as \cite{WilsonSporadicSimple} for $\text{Fi}_{24}'$ and \cite{Monster} for the Monster.
    
Let $n$ be as in Corollary~\ref{cor:SimplificationSimpleGroups}. For $G$ isomorphic to the Mathieu group of degree $12$, $22$, $23$ or $24$ we can set $n=2$ and for sporadic simple groups which are not isomorphic to a Mathieu group, even $n=1$ is sufficient to obtain $\mathcal{J}(G) < 2$. The same applies to the Tits group \cite{GitRepMathematica}. For $G \cong \text{M}_{11}$ one can actually not use Corollary~\ref{cor:SimplificationSimpleGroups}, but the group is small enough that an explicit GAP calculation of $\mathcal{J}(G) = \frac{431}{330} <  2$ is possible \cite{GitRepGAP}. We collect the necessary data in Table \ref{TableSporadic}.
\end{proof}

\begin{table}
\caption{Data for sporadic groups used in Proposition \ref{Sporadics}}
\begin{tabular}{|c|p{6cm}|c|c|}
\hline
Name & Order & Minimal indices & $\ell \leq$ \\ \hline
$\text{M}_{11}$ & 7920 & 11 & 5 \\ \hline
$\text{M}_{12}$ & 95040 & 12,\ 66 & 8 \\ \hline
$\text{M}_{22}$ & 443520 & 22,\ 77 & 7 \\ \hline
$\text{M}_{23}$ & 10200960 & 23,\ 253 & 7 \\ \hline
$\text{M}_{24}$ & 244823040 & 24,\ 276 & 9 \\ \hline
$\text{J}_1$ & 175560 & 266 & 7 \\ \hline
$\text{J}_2$ & 604800 & 100 & 9 \\ \hline
HS & 44352000 & 100 & 10 \\ \hline
$\text{J}_3$ & 50232960 & 6156 & 8 \\ \hline
McL & 898128000 & 275 & 10 \\ \hline
He & 4030387200 & 2058 & 10 \\ \hline
Ru & 145926144000 & 4060 & 15 \\ \hline
Suz & 448345497600 & 1782 & 16 \\ \hline
O'N & 460815505920 & 122760 & 9 \\ \hline
$\text{Co}_3$ & 495766656000 & 276 & 14 \\ \hline
$\text{Co}_2$ & 42305421312000 & 2300 & 11 \\ \hline
$\text{Fi}_{22}$ & 64561751654400 & 3510 & 13 \\ \hline
HN & 273030912000000 & 114000 & 14 \\ \hline
Ly & 51765179004000000 & 8835156 & 9 \\ \hline
Th & 90745943887872000 & 143127000 & 16 \\ \hline
$\text{Fi}_{23}$ & 4089470473293004800 & 31671 & 14 \\ \hline
$\text{Co}_1$ & 4157776806543360000 & 98280 & 22 \\ \hline
$\text{J}_4$ & 86775571046077562880 & 173067389 & 13 \\ \hline
$\text{Fi}_{24}'$ & 1255205709190661721292800 & 306936 & 22 \\ \hline
B & 4154781481226426191177580544000000 & 13571955000 & 30 \\ \hline
M & 808017424794512875886459904961710- & 972394611420091860000 & 43 \\
& 757005754368000000000 & & \\ \hline
Tits & 17971200 & 1600 & 6 \\ \hline
\end{tabular} \label{TableSporadic}
\end{table}

\section{Symmetric and alternating groups}

Let $n$ be a natural number and denote by $A_n$ and $S_n$ the alternating and symmetric group of degree $n$, respectively. A reference for the following discussion is \cite[Theorem 4.8]{Cameron}. Let $M$ be a maximal subgroup of $A_n$. If $M$ acts intransitively on $\Omega = \{1,\ldots,n\}$, then $M$ equals the stabilizer of a proper subset $\Delta$ of $\Omega$ of size strictly less than $n/2$ and $M \cong (S_k \times S_{n-k}) \cap A_n$ where $k=|\Delta|$. In particular, $|M|=k!(n-k)!/2$. If $M$ acts transitively on $\Omega$ but imprimitively, i.e. it stabilizes a nontrivial partition of $\Omega$, then $M$ equals the stabilizer of a partition of $\Omega$ consisting of $b$ blocks of size $a$ each, where $a \cdot b = n$, and $M \cong (S_a \wr S_b) \cap A_n$. Here the wreath product $S_a \wr S_b$ equals the semidirect product $(S_a)^b \rtimes S_b$ with the permutation action on the coordinates. In particular, $|M|=a!^b b!/2$. If $M$ acts transitively and not imprimitively, then we say that it acts primitively. In this case $M$ is described by the O'Nan-Scott Theorem. However, all we need in this case is that $|M| \leq 4^n$ by \cite[Theorem]{PS}. Note that this result does not depend on CFSG. 

\begin{lemma} \label{symn3bound}
Let $n \geq 19$ and let $M$ be a maximal subgroup of $A_n$. Then either $[G:M]=\binom{n}{k}$ with $k \in \{1,2,3\}$ or $[G:M] > n^3/3$.
\end{lemma}

\begin{proof}
If $M$ acts primitively on $\{1,\ldots,n\}$, then $|M| \leq 4^n$ by \cite[Theorem]{PS}. So, if $[A_n:M] \leq n^3/3$, then $n! \leq 4^n \cdot 2 n^3/3$ which implies that $n \leq 15$, a contradiction. Indeed, if $n \geq 11$, then $n \geq 4e$ and Lemma \ref{Stirling} implies that $(n/(4e))^{11} \leq (n/(4e))^n \leq n!/4^n \leq 2n^3/3$ yielding $n \leq 25$. A case by case inspection for $16 \leq n \leq 25$ implies the result \cite{GitRepMathematica}.

If $M$ acts intransitively on $\{1,\ldots,n\}$, then $M \cong (S_k \times S_{n-k}) \cap A_n$ for some $k$ with $1 \leq k < n/2$. The binomial coefficients $[A_n:M] = \binom{n}{k}$ increase for $k=1,2,\ldots,\lfloor n/2 \rfloor$. Therefore, if $\binom{n}{k} \leq n^3/3$ for some $k$ with $4 \leq k \leq \lfloor n/2 \rfloor$, then $\binom{n}{4} \leq n^3/3$, implying $n \leq 13$, a contradiction \cite{GitRepMathematica}. 

Now assume $M$ acts transitively but imprimitively on $\{1,\ldots,n\}$. Then $M \cong (S_a \wr S_b) \cap A_n$ for some integers $a,b > 1$ such that $ab=n$. 
Assume that $[A_n:M]=n!/(a!^b b!) \leq n^3/3$. Stirling's inequality, as formulated in Lemma~\ref{Stirling}, implies 
$$m! \leq \sqrt{2 \pi m} \cdot (m/e)^m \cdot e^{\frac{1}{12m}} < e \cdot m \cdot (m/e)^m$$ 
for all $m \geq 2$. If $m \geq 7$, this follows from $m \geq 2\pi$ and $1<12m$. If $m \leq 6$, this follows by inspection \cite{GitRepMathematica}. Using this and the fact that $a=n/b$ we obtain
\begin{align*}
(n/e)^n & \leq n! \leq (a!^b b!) n^3 < ((ea)(a/e)^a)^b (eb(b/e)^b)n^3 = (a/e)^n (ea)^b (b/e)^b n^3 eb \\
&= (a/e)^n a^b b^b n^3 eb = (1/b)^n (n/e)^n n^{b+3} eb.    
\end{align*}
In other words,
\begin{equation} \label{bn-1}
b^{n-1} < n^{b+3} e.
\end{equation}
Taking natural logarithms
$$(n-1) \ln b < (b+3) \ln n + 1 < (b+4) \ln n.$$
Rearranging,
$$\frac{n-1}{\ln n} < \frac{b+4}{\ln b}$$
The right-hand side is increasing with $b$ if $b \geq 6$ \cite{GitRepMathematica}. So if this is true then, using $b \leq n/2$, we obtain
$$\frac{n-1}{\ln n} < \frac{n/2+4}{\ln(n/2)},$$
implying $n \leq 18$, a contradiction. Now assume $b \leq 5$. If $b=2$, then \eqref{bn-1} becomes $2^{n-1} < e \cdot n^5$ which implies $n \leq 25$, so $n \in \{20,22,24\}$. The inequality $n!/(a!^b b!) \leq n^3/3$ does not hold in this case. If $b=3$, then \eqref{bn-1} becomes $3^{n-1} < e \cdot n^6$ which implies $n \leq 17$. If $b=4$, then \eqref{bn-1} becomes $4^{n-1} < e \cdot n^7$ which implies $n \leq 15$. If $b=5$, then \eqref{bn-1} becomes $5^{n-1} < e \cdot n^8$ which implies $n \leq 15$. We obtain a contradiction in all cases. See \cite{GitRepMathematica} for more details.
\end{proof}

\begin{prop} \label{symfbound}
Let $n \geq 3$ be an integer.
\begin{enumerate}
    \item $\mathcal{J}(S_n) \leq 5/2$ with equality if and only if $n=4$, and $\mathcal{J}(S_n) < 2$ for all $n \geq 7$.
    \item $\mathcal{J}(A_n) \leq 11/6$ with equality if and only if $n=4$, and $\mathcal{J}(A_n) < 4/3$ for all $n \geq 9$. In particular, $A_n$ is HS.
\end{enumerate}
\end{prop}

\begin{proof}
The claim is true for $n \leq 13$ by \cite{GitRepGAP}.
Note that, if $\mathcal{J}(A_n) < 4/3$, then by Lemma \ref{basics}(2), since $S_n/A_n$ is a cyclic group of order $2$, 
$$\mathcal{J}(S_n) \leq \mathcal{J}(A_n) \cdot \mathcal{J}(S_n/A_n) < \frac{4}{3} \cdot \left(1+\frac{1}{2}\right) = 2,$$
so, in order to prove the proposition, it is enough to prove that $\mathcal{J}(A_n) < 4/3$ for all $n \geq 14$. We do this by induction on $n$.

Let $\mathcal{M}$ be a set of representatives of maximal subgroups of $A_n$ up to order. The function $\mathcal{B}$ was defined in (\ref{defh}). We will use the following bounds for $\mathcal{J}(A_n)$, which follow from Lemma \ref{f2}(2) and the fact that $A_{n-1}$ is the unique subgroup of $A_n$ of index $n$, up to isomorphism:
\begin{equation} \label{Anbound}
\mathcal{J}(A_n) \leq 1+\frac{\mathcal{J}(A_{n-1})}{n} + \sum_{\substack{M \in \mathcal{M} \\ M \not \cong A_{n-1}}} \frac{\mathcal{B}(|M|)}{[A_n:M]}.
\end{equation}
We use this to bound $\mathcal{J}(A_n)$ for $n \in \{14,15,16,17,18,19\}$. The strategy is to use the value of $\mathcal{J}(A_{n-1})$ to bound $\mathcal{J}(A_n)$. We know by \cite{GitRepGAP} that $\mathcal{J}(A_{13}) < 1.11$. Using the knowledge of the maximal subgroups of $A_n$, in particular relying on the library of primitive groups in GAP, yields the claim for $n \leq 19$ \cite{GitRepGAP, GitRepMathematica}.

Now assume $n \geq 20$. Let 
$$s = \sum_{\substack{M \in \mathcal{M}, \\ M \mbox{ primitive}}} |M|, \hspace{2cm} 
t= \sum_{\substack{M \in \mathcal{M},\ M \not \cong A_{n-1} \\ M \mbox{ non-primitive}}} |M|.$$ 
Note that as each primitive subgroup acts transitively, its order is divisible by $n$. Let $v$ be the maximal order of a proper primitive subgroup of $A_n$. Then
\begin{equation} \label{sbound}
s \leq \sum_{i \leq v/n} i n \leq n (v/n)^2 = v^2/n.
\end{equation}
The number of conjugacy classes of non-primitive maximal subgroups of $A_n$ is at most
\begin{equation} \label{tbound}
\left| \{k \in \{1,\ldots,n\}\ |\ k < n/2\} \right| + \left| \{a \geq 2\ |\ a\mid n,\ a \neq n\} \right| \leq n/2 + 2\sqrt{n} < n,
\end{equation}
where the last inequality follows from $n \geq 20$.

By Lemma \ref{symn3bound}, which we can use as $n \geq 19$, and as $A_n$ has at most $n$ conjugacy classes of non-primitive maximal subgroups (by \eqref{tbound}) we obtain
$$t \leq w_n := |A_n| \cdot \left( \frac{n}{n^3/3} + \frac{1}{\binom{n}{2}} + \frac{1}{\binom{n}{3}}\right).$$
We claim that 
\begin{equation} \label{s+tbound}
\frac{s+t}{|A_n|} \leq \frac{s+w_n}{|A_n|} \leq \frac{10}{n^2}
\end{equation}
for all $n \geq 20$. This can be checked by GAP if $n \leq 43$ \cite{GitRepGAP}. Now assume that $n \geq 44$, so that $n \geq 16e$. We use $v \leq 4^n$ \cite[Theorem]{PS}, which implies $s \leq 16^n/n$ (by \eqref{sbound}). If $44 \leq n \leq 52$ then the bound $t \leq w_n$ implies \eqref{s+tbound} \cite{GitRepMathematica}. Now assume $n \geq 53$. Since $n! \geq (n/e)^n$, we have $16^n/n! \leq (16e/n)^n \leq (16e/n)^{44}$ and this is less than $1/n^2$ for $n \geq 53$ \cite{GitRepMathematica}. We deduce that $s/|A_n| \leq (2/n)(16e/n)^{44} \leq 2/n^3$  and hence
\begin{align*}
\frac{s+t}{|A_n|} & \leq \frac{2}{n^3} + \frac{3}{n^2} + \frac{2}{n(n-1)} + \frac{6}{n(n-1)(n-2)}
\end{align*}
This is less than $10/n^2$ \cite{GitRepMathematica}. The claim \eqref{s+tbound} follows.

Moreover, we observe that, since $\ln(\ln(n)) \leq \frac{1}{2} \ln(n)$,
\begin{equation} \label{hn!bound}
\mathcal{B}(n!/2) \leq e^{\gamma} \ln \ln (n^n) + 0.7 = e^{\gamma} \ln(n \ln(n)) + 0.7 \leq \frac{3}{2} e^{\gamma} \ln(n) + 0.7.
\end{equation}

By induction, $\mathcal{J}(A_{n-1}) < 4/3$. Using first \eqref{Anbound} and then \eqref{s+tbound} and \eqref{hn!bound}, we have (see \cite{GitRepMathematica})
\begin{align}\label{BoundJAn}
\mathcal{J}(A_n) & \leq 1 + \frac{\mathcal{J}(A_{n-1})}{[A_n:A_{n-1}]} + \frac{\mathcal{B}(|A_n|)}{|A_n|} \cdot (s+t) \\ 
& < 1+ \frac{4}{3n} + \left( \frac{3}{2} e^{\gamma} \ln(n) + 0.7 \right) \cdot \frac{10}{n^2} < \frac{4}{3}. \nonumber
\end{align}
The proof is completed.
\end{proof}

\section{Simple groups of Lie type}

In this section we will prove that $G$ is HS for $G$ a finite simple group of Lie type, by showing that $\mathcal{J}(G) < 2$ in this case. Throughout the section $G$ will be a simple group of Lie type. We first record the necessary bounds on the number of maximal subgroups and the minimal index of a maximal subgroup from the literature for the classical case. For us, the \emph{parameters} of a simple classical group of Lie type are the dimension $n$ of the underlying vector space (so that for example $n$ is even, in the symplectic case) and the size $q$ of the underlying field, except for the unitary case in which the underlying field has size $q^2$.

The following is Theorem 4.1 in \cite{Liebeck}.

\begin{lemma} \label{Lemmaq3n}
Let $M$ be a maximal subgroup of a finite simple classical group $G$, with parameters $n$, $q$, and let $\delta$ be equal to $2$ if $G=\PSU(n,q)$ and $1$ otherwise. Then one of the following holds.
\begin{enumerate}
\item $M$ lies in one of the Aschbacher classes $\mathscr{C}_i$, $i=1,\ldots,8$.
\item $M$ is isomorphic to one of the groups $A_{n+1}$, $S_{n+1}$, $A_{n+2}$, $S_{n+2}$.
\item $|M| < q^{3 \delta n}$.
\end{enumerate}
\end{lemma}

For a natural number $n$ let $d(n)$ be the number of divisors of $n$ and let $\pi(n)$ be the number of prime divisors of $n$.

\begin{lemma} \label{Lemmac1c8}
If $G$ is a finite simple classical group with parameters $n,q$, then the number of isomorphism classes of maximal subgroups of $G$ in the Aschbacher classes $\mathscr{C}_1,\ldots,\mathscr{C}_8$ is at most
$$\frac{3}{2}n + 4d(n) + \pi(n) + \log_2 \log_2 q + 3 \log_2 n + 8.$$
\end{lemma}
\begin{proof}
This follows from \cite[Lemma 2.1]{GKS} and the obvious fact that, in the notation of \cite{GKS}, $\Delta$-conjugate subgroups are isomorphic.
\end{proof}

\begin{lemma} \label{LemmaImprovedL}
Let $G$ be a finite simple classical group with parameters $n,q$. Let $\delta$ be $2$ if $G=\PSU(n,q)$ and $1$ otherwise. Let $m_1$ be the smallest index of a proper subgroup of $G$. Let $k$ be the number of isomorphism types of maximal subgroups of $G$ of order at least $q^{3\delta n}$. Let $\ell = |\{|M|\ |\ M <_{\max} G,\ |M| \leq q^{3\delta n}\}|$. Then
\begin{align*}
\mathcal{J}(G)-1 & \leq \left( \frac{k}{m_1} + \frac{q^{3\delta n} \cdot \ell}{|G|} \right) \cdot (e^{\gamma}+1) \cdot \ln\ln \left( \frac{|G|}{m_1} \right).
\end{align*}
Moreover, $k \leq (3/2)n+14\sqrt{n}+12+\log_2\log_2 q$.
\end{lemma}

\begin{proof}
Let $\mathcal{M}$ be a set of representatives of maximal subgroups of $G$ up to isomorphism and let $M_1,\ldots,M_k$ be the members of $\mathcal{M}$ whose order is at least $q^{3\delta n}$. If $\mathcal{M}$ is non-empty we assume $[G:M_1] = m_1$. We apply Lemma \ref{f2} with $\mathcal{O} = \{|M|\ |\ M \in \mathcal{M},\ M \not \in \{M_1,\ldots,M_k\}\}$. Note that $\mathcal{O} \subseteq \{1,\ldots,q^{3\delta n}\}$ and $|\mathcal{O}| \leq \ell$, hence the sum of the orders of the elements of $\mathcal{O}$ is at most $q^{3\delta n} \cdot \ell$. Choose $j \in \{1,\ldots,k\}$ so that $\mathcal{J}(M_j)$ is largest. Since $[G:M_i] \geq m_1$ for all $i=1,\ldots,k$,
\begin{align*}
\mathcal{J}(G)-1 & \leq \sum_{i=1}^k \frac{\mathcal{J}(M_i)}{[G:M_i]} + \sum_{m \in \mathcal{O}} \frac{\mathcal{B}(m)}{|G|/m} \leq k \cdot \mathcal{J}(M_j)/m_1 + |\mathcal{O}| \frac{\mathcal{B}(q^{3\delta n})}{|G|/q^{3\delta n}} \\ 
& \leq \left(\frac{k}{m_1} + \frac{q^{3\delta n} \cdot \ell}{|G|} \right) \cdot (e^{\gamma}+1) \cdot \ln \ln \left( \frac{|G|}{m_1} \right).
\end{align*}
We used the fact that $\mathcal{B}(x)$ is increasing for $x \geq 6.23$ to obtain $\mathcal{B}(|M|) \leq \mathcal{B}(|G|/m_1)$ for each $M \in \mathcal{M}$ and also that $\mathcal{J}(M_j) \leq \mathcal{B}(|M_j|)$ by Lemma~\ref{f2}. Note here that $\mathcal{B}(6) < \mathcal{B}(7)$, so that this inequality follows from Lemma~\ref{maxs3}. The first claim of the lemma follows. 

The upper bound for $k$ follows from Lemma \ref{Lemmaq3n}, Lemma \ref{Lemmac1c8} and $d(n) \leq 2\sqrt{n}$, $\pi(n) \leq 2\sqrt{n}$, $3\log_2(n) \leq 4\sqrt{n}$.
\end{proof}

\begin{lemma} \label{Lemma5.2}
Let $G$ be a simple classical group with parameters $n,q$ and let $m_1$ be the smallest index of a proper subgroup of $G$. Then
$$\mathcal{J}(G)-1 \leq \frac{1}{m_1} \left( 2n^{5.2}+n \log_2 \log_2 q \right) (e^{\gamma}+1) \ln \ln \left( \frac{|G|}{m_1} \right).$$
\end{lemma}
\begin{proof}
It follows from Lemma \ref{f2}, with $\mathcal{O}=\emptyset$ and the fact that $k \leq 2n^{5.2} + n\log_2\log_2 q$ by \cite[Theorem 1.1]{Hasa}.
\end{proof}

Recall the generic isomorphisms between group of Lie type:
$\PSU(2,q) \cong \PSL(2,q)$, 
$\PSp(2,q) = \PSL(2,q)$, 
$\PSp(4,q) \cong \OO(5,q)$, 
$\OO(3,q) \cong \PSL(2,q)$, 
$\OO^{+}(4,q) \cong \PSL(2,q) \times \PSL(2,q)$, 
$\OO^{-}(4,q) \cong \PSL(2,q^2)$, 
$\OO^{+}(6,q) \cong \PSL(4,q)$, 
$\OO^{-}(6,q) \cong \PSU(4,q)$.
We also have $\OO(n,q) \cong \PSp(n-1,q)$, if $n$ is odd and $q$ is even. In view of these isomorphisms we restrict our choices of parameters in the proofs of the following Lemma~\ref{Bounds} and Proposition~\ref{Classical}. Namely, we only consider the group with smaller dimension parameter $n$ and if these parameters are equal we prefer projective special linear over other types of groups.

\begin{lemma} \label{Bounds}
Let $G$ be a simple classical group with parameters $n,q$ and let $m_1$ be the smallest index of a proper subgroup of $G$. Then $|G|$ and $m_1$ are bounded according to Table \ref{TableGm1}. We deduce that, in all cases,
$$q^{n^2} \geq |G| \geq \frac{1}{4} (q-1)^{(n-1)/2} q^{n(n-2)/2}, \hspace{.6cm}
m_1(G) \geq q^{n-2}.$$
\end{lemma}
\begin{proof}
The orders of the groups are given in \cite[Table 24.1]{MalleTesterman} and the values for $m_1$ were computed by Cooperstein and are compiled, with some corrections obtained by Mazurov and Vasil'ev, in \cite[Table 4]{GMPS}. We note that the notation in \cite[Table 4]{GMPS} for the dimension is different from ours in the case of symplectic groups.
We use the inequality $q^k-1 \geq (q-1)q^{k-1}$ from which we deduce
\begin{align*}
\prod_{i=2}^k (q^i-1) & \geq \prod_{i=2}^k ((q-1)q^{i-1}) = (q-1)^{k-1} q^{(k^2-k)/2}, \\
\prod_{i=1}^k (q^{2i}-1) & \geq \prod_{i=1}^k ((q-1)q^{2i-1}) = (q-1)^k q^{k(k+1)-k} = (q-1)^k q^{k^2}.
\end{align*}
These are used to obtain the lower bounds for $|G|$ in Table \ref{TableGm1}.
\end{proof}

\begin{table}[ht]
\caption{Bounds on orders and minimal index for classical groups}
\begin{tabular}{|c|c|c|c|} \hline
\mbox{Group} & $|G| \leq$ & $|G| \geq$ & $m_1(n,q) \geq$ \\ \hline
$\PSL(n,q)$ & $q^{n^2}$ & $(q-1)^{n-2} q^{n^2-n}$ & $q^{n-1}$ \mbox{ except} \\
$n \geq 2$,\ $(n,q) \neq (2,2),(2,3)$ & & & $m_1(2,9)=6$ \\ \hline
$\PSU(n,q)$ & $q^{n^2}$ & $\frac{(q-1)^{n-1}}{q+1} q^{n^2-n}$ & $\frac{q^{2n-2}}{q+1}$ \mbox{ except } \\
$n \geq 3$,\ $(n,q) \neq (3,2)$ & & & $m_1(3,5)=50$, \\
& & & $m_1(4,q)=(q+1)(q^3+1)$ \\ \hline
$\PSp(n,q)$ & $q^{n^2}$ & $\frac{1}{2} (q-1)^{n/2} q^{n^2/2}$ & $2^{n-2}$ \mbox{ if } $q=2$, \\
$n \geq 4$\ \mbox{even},\ $(n,q) \neq (4,2)$ & & & $q^{n-1}$ \mbox{ if } $q > 2$ \\ \hline
$\OO(n,q)$ & $q^{n^2/2}$ & $\frac{1}{2} (q-1)^{(n-1)/2} q^{(n-1)^2/2}$ & $q^{n-2}$ \\
$n \geq 7$\ \mbox{odd},\ $q$\ \mbox{odd} & & & \\ \hline
$\OO^{\pm}(n,q)$ & $q^{n^2/2}$ & $\frac{1}{4} (q-1)^{n/2} q^{n(n-2)/2}$ & $q^{n-2}$ \\
$n \geq 8$ \mbox{ even} & & & \\ \hline
\end{tabular} \label{TableGm1}
\end{table}

\begin{prop}\label{Classical}
If $G$ is a finite simple classical group, then $\mathcal{J}(G)<2$. In particular $G$ is HS.
\end{prop}

\begin{proof}
Let $G$ be a finite simple classical group of Lie type with parameters $n,q$ and let $\delta$ be $2$ if $G=\PSU(n,q)$ and $1$ otherwise. Assume first that $n \geq 13$.
By Lemma \ref{LemmaImprovedL}
\begin{equation} \label{fg-1leq}
\mathcal{J}(G)-1 \leq \left( \frac{q^{3\delta n} \cdot \ell}{|G|} + \frac{(3/2)n+14\sqrt{n}+12+\log_2\log_2 (q)}{m_1(G)} \right) (e^{\gamma}+1) \ln \ln(q^{n^2})
\end{equation}
By \eqref{fg-1leq} and Lemma \ref{Bounds}, since $\ell \leq q^{3\delta n}$, $\log_2 \log_2 (q) \leq q$ and \begin{equation}
\ln(\ln(q^{n^2})) = \ln(n^2 \ln(q)) = 2\ln(n)+\ln\ln(q) \leq 2 \sqrt{q} \ln(n),
\end{equation}
we deduce that 
\begin{align}\label{BoundJClassicalqBigger2nBig}
\mathcal{J}(G)-1 & \leq \left( \frac{4 \cdot q^{6\delta n}}{q^{n(n-2)/2}} + \frac{(3/2)n + 14\sqrt{n} + 12 + q}{q^{n-2}} \right) \cdot 3 \cdot 2 \sqrt{q} \ln(n) \\
 & \leq \frac{\left( (3/2)n+14\sqrt{n}+16 + q \right) \cdot 6 \sqrt{q} \ln(n)}{\min\{q^{n(n-2)/2-6\delta n},q^{n-2}\}} \nonumber \leq \frac{6nq \cdot 6 \sqrt{q} \ln(n)}{\min\{q^{n(n-2)/2-6\delta n},q^{n-2}\}}
 \nonumber
\end{align}

We separate several cases. For the specific computations and solutions of inequalities, we refer to \cite{GitRepMathematica}. Assume first that $n \geq 13$. If $G$ is not of unitary type, so that $\delta=1$, then we have the following cases.

\begin{itemize}
    \item $n \geq 16$. Since $n(n-2)/2-6n \geq n-2$ and $q \geq 2$, we deduce that $\mathcal{J}(G)-1 \leq 36 n^2/q^{n-7/2}$ by \eqref{BoundJClassicalqBigger2nBig} and this is less than $1$ except for the case $(n,q)=(16,2)$.
    \item $15 \leq n \leq 16$, $q=2$. We apply \eqref{fg-1leq} with $\ell \leq q^{3n}$.
    \item $13 \leq n \leq 14, q = 2$. We use Lemma \ref{LemmaImprovedL}, where $\ell$ is bounded above by the number of divisors of $|G|$, computed in \cite{GitRepGAP}, and $m_1,|G|$ are bounded by Lemma \ref{Bounds}. 
\item $n=15, q \geq 3$. In this case \eqref{BoundJClassicalqBigger2nBig} gives 
$$\mathcal{J}(G)-1 \leq ((3/2) 15 + 14 \sqrt{15} + 16 + q) \cdot 6 \ln(15) \cdot q^{-7}$$ which is less than $1$.
    \item $n=14$, $q \geq 3$. In this case, bounding $\ell$ above by $q^{3n}$ and using Lemmas \ref{Lemmac1c8} and \ref{Bounds},
\begin{equation} \label{14case} \mathcal{J}(G)-1 \leq \left( \frac{4}{(q-1)^{13/2}} + \frac{63+\log_2\log_2(q)}{q^{12}} \right) (e^{\gamma}+1) \ln\ln(q^{14^2}).
\end{equation}
This is less than $1$. If $q \in \{3,4,5,7,8\}$ it can be checked by case-by-case inspection, and if $q \geq 9$ then $q-1 \geq q^{9/10}$ and, using $\log_2\log_2(q) \leq q$, $e^{\gamma}+1 < 2.8$ and $\ln\ln(q^{14^2}) \leq 14q$, \eqref{14case} is at most $68 \cdot 2.8 \cdot 14 \cdot q^{2-9 \cdot 13/20} < 1$.
    \item $n=13, q \geq 3$. Lemmas \ref{Lemma5.2} and \ref{Bounds} imply that
\begin{equation}\label{BoundJClassicalqBigger2n13}
\mathcal{J}(G)-1 \leq \frac{1}{q^{11}} \left( 2 \cdot 13^{5.2} + 13 \log_2 \log_2 (q) \right) \cdot 2.8 \cdot \ln \left( \ln \left( q^{13^2} \right) \right)
\end{equation}
which is smaller than $1$ for $q \geq 5$. This can be achieved with Mathematica by using the bounds $\log_2 \log_2 (q) \leq q$, $\ln(\ln(q^{13^2})) \leq 13q$.  If $q \in \{3,4\}$ we use Lemma \ref{LemmaImprovedL} where the number $\ell$ can be bounded above by the number of divisors of $|G|$, obtained in \cite{GitRepGAP}, and $m_1$ is bounded by Lemma \ref{Bounds}. We deduce that $\mathcal{J}(G) < 2$.
\end{itemize}

Assume now that $G=\PSU(n,q)$, so that $\delta=2$, and $n \geq 13$, where $F=\mathbb{F}_{q^2}$ is the base field. We have $|G| \geq q^{n^2-n}$ and $m_1 \geq q^{2n-3}$. Arguing as in \eqref{BoundJClassicalqBigger2nBig} gives, if $n \geq 14$,
$$\mathcal{J}(G)-1 \leq \frac{6nq \cdot 6 \sqrt{q} \ln(n)}{\min\{q^{n^2-n-12 n},q^{2n-3}\}} = \frac{36 n \ln(n)}{q^{\min\{n^2-13n-3/2,2n-9/2\}}} < 1.$$

Assume $n=13$ and $q \geq 3$. We use Lemma \ref{Lemma5.2} and $\ln\ln(q^{n^2}) \leq 2 \sqrt{q} \ln(n)$, which gives
\begin{align*}
\mathcal{J}(G)-1 & \leq \frac{2.8 \cdot 2 \sqrt{q} \ln(13) \cdot (q^2-1) \cdot (2 \cdot 13^{5.2}+13 \log_2\log_2(q))}{(q^{13}+1)(q^{12}-1)} \\ & \leq \frac{2.8 \cdot 2 \cdot \ln(13) \cdot (2 \cdot 13^{5.2} + 13q)}{q^{21}} < 1.
\end{align*}

Finally, assume $(n,q)=(13,2)$. Then
$$|G|=302002740016633758616955204618296772787569688576000.$$
This number has $6552576$ positive divisors. Using this as an upper bound for $\ell$, Lemma \ref{LemmaImprovedL} again gives $\mathcal{J}(G) < 2$.

Now assume $n \leq 12$. Write $q=p^k$, $p$ prime. Let $\pi(k)$ denote the number of prime divisors of $k$ and let $\ell$ be the number of isomorphism types of maximal subgroups of $G$. We have the bound $\ell \leq C + \pi(k)$ where the constant $C$ appears in Table \ref{TableC} and it can be easily deduced from \cite[Tables 8.1-8.85]{BHR}. Recall that the orthogonal groups in dimension $2$ are not simple, the groups $\PSp$, $\OO^+$, $\OO^-$ are only defined in even dimension and $\OO$ is only defined in odd dimension. Also, recall the generic isomorphisms listed before Lemma \ref{Bounds} which allow us to include more ``$*$'' in Table~\ref{TableC}.
\begin{table}[ht]
\caption{Constant $C$ - upper bound for $\ell-\pi(k)$}
\begin{tabular}{|c||c|c|c|c|c|c|c|c|c|c|c|} \hline
 $n=$ &  $2$ & $3$ & $4$ &  $5$ &  $6$ &  $7$ &  $8$ &  $9$ &  $10$ & $11$ & $12$ \\ \hline \hline
$\PSL$ & $5$ & $9$ & $14$ & $12$ & $23$ & $14$ & $20$ & $21$ & $29$ & $21$ & $31$ \\ \hline
$\PSU$ & $*$ & $8$ & $13$ & $10$ & $20$ & $11$ & $17$ & $16$ & $25$ & $16$ & $26$ \\ \hline
$\PSp$ & $*$ & $*$ & $9$ & $*$ & $17$ & $*$ & $18$ & $*$ &  $16$ & $*$ &  $24$ \\ \hline
$P\Omega$ & $*$ & $*$ & $*$ & $*$ & $*$ & $12$ & $*$ & $21$ &  $*$ &  $19$ & $*$ \\ \hline
$P\Omega^+$ & $*$ & $*$ & $*$ & $*$ & $*$ & $*$ & $22$ & $*$ & $20$ & $*$ &  $31$ \\ \hline
$P\Omega^-$ & $*$ & $*$ & $*$ & $*$ & $*$ & $*$ & $10$ & $*$ & $19$ & $*$ &  $21$ \\ \hline
\end{tabular}\label{TableC}
\end{table}

Table \ref{TableC} implies that $\ell \leq 31+\pi(k)$ for all these groups. If $k$ has $u$ distinct prime divisors, then clearly $k \geq 2^u$, so $u \leq \log_2 k$. Since $q=p^k \geq 2^k$ we have $k \leq \log_2 q$. Therefore $\pi(k) = u \leq \log_2 k \leq \log_2 \log_2 q$ and hence $\ell \leq 31+\log_2 \log_2 q$.

Let $m_1$ be the smallest index of a maximal subgroup of $G$.
We will use the bounds from Lemma \ref{Bounds} and Corollary~\ref{cor:SimplificationSimpleGroups} with $\mathcal{B}_1$:
\begin{equation}\label{BoundsnSmall}
\mathcal{J}(G)-1 \leq (\ell/m_1) \cdot (e^{\gamma}+1) \cdot \ln\ln(|G|/m_1)
\end{equation}
We will consider each family of groups separately. For the specific computations and solutions of inequalities, we refer to \cite{GitRepMathematica}.

\begin{itemize}
\item Let $G=\PSL(n,q)$. If $(n,q)=(2,9)$, then $G$ is isomorphic to $A_6$ and hence $\mathcal{J}(G)<2$ by Proposition \ref{symfbound}. Now assume that $(n,q) \neq (2,9)$. By Lemma \ref{Bounds}, $m_1 \geq q^{n-1}$ so
\begin{equation*}
\mathcal{J}(G)-1 < \frac{C+\log_2\log_2(q)}{q^{n-1}} \cdot 2.8 \cdot \ln\ln(q^{n^2-n+1}).
\end{equation*}
For purposes of simplification, we will sometimes use the rough bounds $\log_2\log_2(q) \leq \log_2(q)$ and $\ln\ln(q^{n^2-n+1}) \leq 2 \log_2(n) \log_2(q)$ (see \cite{GitRepMathematica} for the details).
The following table summarizes, for each $n$, the corresponding value of $C$ above and the smallest value of $q$ that guarantees that $\mathcal{J}(G)< 2$ via the above bound.
\begin{table}[ht]
\caption{Inequality values of $q$ for $\PSL$}
\begin{tabular}{|c|c|c|c|c|c|c|c|c|c|c|c|} \hline
$n=$ & $2$ & $3$ & $4$ & $5$ & $6$ & $7$ & $8$ & $9$ & $10$ & $11$ & $12$ \\ \hline
$C=$ & $5$ & $9$ & $14$ & $12$ & $23$ & $14$ & $20$ & $21$ & $29$ & $21$ & $31$ \\ \hline 
$q \geq$ & $59$ & $11$ & $7$ & $4$ & $3$ & $3$ & $3$ & $2$ & $2$ & $2$ & $2$ \\ \hline
\end{tabular}\label{TablePSL}
\end{table}

A GAP computation shows that $\mathcal{J}(G)<2$ for $(n,q)$ in the following list: $(3,2)$, $(3,3)$, $(3,4)$, $(3,5)$, $(3,7)$, $(3,8)$, $(3,9)$, $(4,2)$, $(4,3)$, $(5,2)$ \cite{GitRepGAP}.

We use Corollary~\ref{cor:SimplificationSimpleGroups}, with the function $\mathcal{B}_1$, for the remaining cases when $n > 2$. We list a bound for $\ell$ deduced from \cite[Tables 8.8, 8.9, 8.18, 8.19, 8.24, 8.25, 8.35, 8.36, 8.44, 8.45]{BHR} and $m_1=(q^n-1)/(q-1)$ \cite[Table 4]{GMPS}.
If $(n,q)=(4,4)$ then $\ell \leq 9$, $m_1=85$;
if $(n,q)=(4,5)$ then $\ell \leq 11$, $m_1 = 156$;
if $(n,q)=(5,3)$ then $\ell \leq 10$, $m_1 = 121$;
if $(n,q)=(6,2)$ then $\ell = 7$, $m_1 = 63$;
if $(n,q)=(7,2)$ then $\ell \leq 10$, $m_1=127$;
if $(n,q)=(8,2)$ then $\ell \leq 13$, $m_1 = 255$.
It follows that $\mathcal{J}(G) < 2$ in all these cases.

Now assume $n=2$. The subgroups of $G=\PSL(2,q)$ were classified by Dickson already in 1900. They are listed in \cite[Chapter II, Hauptsatz 8.27]{Huppert}. The maximal ones can easily be distilled from there. Alternatively, one can use \cite[Theorem 3.8]{Wilson} for $n=1$. We will provide just some knowledge we need. Let $q = p^k$ be a prime power, $d = \gcd(p-1, 2)$ and $G = \PSL(2,q)$. Recall that $|\PSL(2,q)|=(q^3-q)/d$ and $|\PGL(2,q)|=q^3-q$. If $M$ is a maximal subgroup of $G$, then it is isomorphic to one of the following:
\begin{enumerate}
    \item $\PSL(2, p^{k/r})$ where $r$ is a prime divisor of $k$ distinct from $k$.
    \item $\PGL(2, p^{k/2})$ where $k$ is even.
    \item A dihedral group $D_{2(q-1)/d}$ or $D_{2(q+1)/d}$.
    \item A Frobenius group $C_p^k \rtimes C_{(q-1)/d}$
    \item $A_4$, $S_4$ or $A_5$ (only under some additional conditions).
\end{enumerate}

If $q \leq 13$ then $\mathcal{J}(G) < 2$ by \cite{GitRepGAP}. Assume $q \geq 16$, so that the biggest maximal subgroup of $G$ is the Frobenius group $C_p^k \rtimes C_{(q-1)/d}$ (by \cite[Table 4]{GMPS} for example, or by simple comparison). According to whether $k$ is odd or even, the maximal subgroup of second maximal size is isomorphic to $D_{2(q+1)/d}$, which has index $q(q-1)/2$, or $\PGL(2,\sqrt{q})$, which has index $(q^3-q)/((q^{3/2}-q^{1/2})d)$. We use Corollary~\ref{cor:SimplificationSimpleGroups} with $\mathcal{B}_2$, $m_1=q+1$, $\ell \leq 5+\log_2 \log_2 q$. 

Let $m_2$ denote the second minimal index of a maximal subgroup of $G$, so that there are the following two possibilities. If $m_2=q(q-1)/2$, then
$$\mathcal{J}(G)-1 < 2.8 \left( \frac{\ln\ln(q(q-1))}{q+1} + \frac{4+\log_2 \log_2(q)}{q(q-1)/2} \ln \ln \left( 2(q+1) \right)\right) < 1$$
If $m_2 = (q^3-q)/((q^{3/2}-q^{1/2})d)$, then
$$\mathcal{J}(G)-1 < 2.8 \left( \frac{\ln\ln(q(q-1))}{q+1} + \frac{4+\log_2 \log_2(q)}{(q^3-q)/((q^{3/2}-q^{1/2})d)} \ln \ln \left( q^{3/2}-q^{1/2} \right)\right) < 1.$$
These inequalities involve a decreasing function, so it is enough to check them for small values of $q$.

\item For $G=\PSU(n,q)$, $n \in \{3,\ldots,12\}$, we use the bound
$$\mathcal{J}(G)-1 < \frac{C+\log_2\log_2(q)}{m_1} \cdot 2.8 \cdot \ln \ln (q^{n^2}).$$
We bound $m_1$ using Table \ref{TableGm1}. Using that the right-hand side is decreasing, this is enough to show that $\mathcal{J}(G)<2$, except if $(n,q)$ is one of $(3,3)$, $(3,4)$, $(3,5)$, $(4,2)$. In these cases the claim follows by \cite{GitRepGAP}.

\item For $G=\PSp(n,q)$, $n \in \{4,6,8,10,12\}$, we use the bound
$$\mathcal{J}(G)-1 < \frac{C+\log_2\log_2(q)}{q^{n-\varepsilon}} \cdot 2.8 \cdot \ln \ln (q^{n^2}),$$
where $\varepsilon=1$, if $q>2$, and $\varepsilon=2$, if $q=2$. This is enough to show that $\mathcal{J}(G)<2$, except if $(n,q) \in \{(4,3),(4,4),(6,2),(8,2)\}$ (recall that $(n,q) \neq (4,2)$ as $\PSp(4,2) \cong S_6$ is not simple). If $(n,q) \in\{(4,3), (4,4), (6,2)\}$, then $\mathcal{J}(G)<2$ by \cite{GitRepGAP}. If $(n,q)=(8,2)$, we use Corollary~\ref{cor:SimplificationSimpleGroups} with $\mathcal{B}_1$, $\ell \leq 11$ by \cite[Tables 8.48, 8.49]{BHR}, $m_1 = 120$, to obtain $\mathcal{J}(G)<2$.

\item For $G$ orthogonal of type $\OO(n,q)$ ($n \geq 7$ odd, $q$ odd), $\OO^{\pm}(n,q)$ ($n \geq 8$ even), $m_1 \geq q^{n-2}$ by Lemma \ref{Bounds}. So,
$$\mathcal{J}(G)-1 < \frac{C + \log_2\log_2(q)}{q^{n-2}} \cdot 2.8 \cdot \ln \ln(q^{n^2/2-n+2}).$$
This is enough to show that $\mathcal{J}(G)<2$, except if $(n,q)=(8,2)$. In this case $\mathcal{J}(\OO^{\pm}(8,2))<2$ by \cite{GitRepGAP}.
\end{itemize}
The proof is complete.
\end{proof}

We next handle exceptional groups of Lie type. The subfield maximal subgroups of $\text{E}_8(q)$ are the ones of type $\text{E}_8(q_0)$ where $q_0=q^{1/r}$ for $r$ a prime divisor of $a$, where $q=p^a$. The proof of the following proposition was provided to us by Martin Liebeck.

\begin{prop} \label{E8lbound}
Let $G = \text{E}_8(q)$. The number of isomorphism types of non-subfield maximal subgroups of $G$ is at most $166$.
\end{prop}

\begin{proof} 
Let $q = p^a$, where $p$ is prime. Denote by $\mbox{Lie}(p)$ the finite simple groups of Lie type over fields of characteristic $p$ and by $\mbox{Lie}(p')$ the finite simple groups of Lie type over fields of characteristic different from $p$. Regard $G$ as $\overline G^F$, where $\overline G$ is a simple algebraic group of type $\text{E}_8$ over $\overline{\mathbb{F}}_p$ and $F$ is a Frobenius endomorphism. According to \cite[Theorem 8]{LiebeckSeitzSurvey}, the non-subfield maximal subgroups of $G$ fall into the following families:
\begin{itemize}
\item[(1)] maximal parabolic subgroups, these correspond to the nodes of the Dynkin diagram;
\item[(2)] subgroups $N_G(\overline M^F)$, where $\overline M$ is reductive of maximal rank: the possibilities are listed in \cite[Tables 5.1, 5.2]{LiebeckSaxlSeitz92};
\item[(3)] subgroups $N_G(\overline M^F)$, where $\overline M$ is reductive of non-maximal rank: these are listed in \cite[Table 3]{LiebeckSeitzSurvey}, together with the subgroup $\PGL_2(q) \times \mbox{Sym}_5$;
\item[(4)] `exotic local' subgroups $2^{5+10}.\SL(5,2)$, $5^3.\SL(3,5)$ (see \cite[Section 2]{LiebeckSeitzSurvey}), and the `Borovik' subgroup $(\text{Alt}_5 \times \text{Alt}_6).2^2$;
\item[(5)] the class ${\mathcal U}$ of almost simple subgroups with socle $S$, not occurring in items (1)-(4); we divide these into subclasses:
\begin{itemize}
\item[(a)] ${\mathcal U}_{p,1}$: $S \in \mbox{Lie}(p),\,S \neq \PSL(2,p^a)$
\item[(b)] ${\mathcal U}_{p,2}$: $S \in \mbox{Lie}(p),\,S = \PSL(2,p^a)$
\item[(c)] ${\mathcal U}_{\mbox{Alt}}$: $S$ alternating
\item[(d)] ${\mathcal U}_{\mbox{Spor}}$: $S$ sporadic
\item[(e)] ${\mathcal U}_{p'}$: $S \in \mbox{Lie}(p')$.
\end{itemize}
\end{itemize}

Let $N_i$ be the number of isomorphism types of maximal subgroups in class (i), for $i=1,\ldots,5$. By inspection, 
\begin{equation}\label{ins}
N_1 = 8,\;N_2\leq 29 ,\;N_3\leq 7,\;N_4\leq 3.
\end{equation}
To bound $N_5$, we distinguish between different defining characteristics and write $N_{5}(p)$. We separate the five cases given above as $N_5(p) = N_{5a}(p)+\cdots + N_{5e}(p)$, where $N_{5a}(p)$ is the number of isomorphism types in characteristic $p$ for the class (5a), and so on.

By \cite[Theorem 1.1]{CravenMediumRank}, the possibilities for $S$ in class (5a) are
\[
\,\PSL(3,4), \, \PSU(3,4),\,\PSU(3,8),\,\PSU(4,2),\,^2\!B_2(8),\ \PSL(3,3), \,\PSU(3,3).
\]

In \cite[Theorem 1.1]{CravenMediumRank} the group $\Sp_4(2)'$ also appears, but it is isomorphic to $A_6$ and hence falls in another class for us. The defining characteristic is $p$ in the first five cases above, so that $p=2,2,2,2,2,3,3$ respectively. 
Counting the almost simple groups with these socles, we obtain
\[
N_{5a}(2) \leq 26, \ N_{5a}(3) \leq 4, \ N_{5a}(p) =0 \text{ else}.
\]

Now consider class (5b), where $S = \PSL(2,p^a)$. We use \cite[Theorem 6]{LiebeckSeitz98} (and remark 3 after its statement). Note that case (ii) of this theorem is contained in our classes (2) and (3). Therefore $p^a \leq t(\Sigma(\text{E}_8)) \cdot (2,p-1)$ where $t(\Sigma(\text{E}_8))$ is a constant defined in terms of the root system $\Sigma(\text{E}_8)$ (see \cite[Definition on page 3410]{LiebeckSeitz98}); and in fact $t(\Sigma(\text{E}_8)) = 1312$ by \cite[Theorem 2]{Lawther}.
Counting the almost simple groups with socle $S$, excluding the socles $\PSL(2,4) \cong \PSL(2,5) \cong A_5$ and $\PSL(2,9) \cong A_6$ which lie in class (5c), we obtain
\begin{align*}
   N_{5b}(2) &\leq 24, \ N_{5b}(3) \leq 30, \ N_{5b}(5) \leq 17, \ N_{5b}(7) \leq 19, \ N_{5b}(p) \leq 11 \ (11 \leq p \leq 13), \\ \ N_{5b}(p) &\leq 7 \ (17 \leq p \leq 47), \ N_{5b}(p) \leq 2 \ (53 \leq p \leq 2617), \ N_{5b}(p) =0 \text{ else}.
\end{align*}

For class (5c), the main result of \cite{CravenAlternating} implies that $S$ is $A_6$ or $A_7$. 
If $p=2$, then only $A_6$ occurs and if $p=5$ then only $A_7$ occurs. Hence,
\[
N_{5c}(2) \leq 5, \ N_{5c}(5) \leq 2, \ N_{5c}(p) \leq 7 \text{ else.} 
\]

For class (5d), the sporadic simple subgroups of $\text{E}_8$ are listed in \cite[Table 10.2]{LiebeckSeitz99}; they are
\[
\text{M}_{11},\ \text{M}_{12}\,(p=2,5),\ \text{J}_1\,(p=11),\ \text{J}_2\,(p=2),\ \text{J}_3\,(p=2),\ \text{Th}\,(p=3).
\]
So,
\[
N_{5d}(2) \leq 7, \ N_{5d}(3) \leq 2, \ N_{5d}(5) \leq 3, \ N_{5d}(11) \leq 2, \ N_{5d}(p) = 1 \text{ else.}
\]

Finally consider class (5e). The simple groups $S \in \mbox{Lie}(p')$ contained in $\text{E}_8$ are given in \cite[Tables 10.3, 10.4]{LiebeckSeitz99}. They are:
\[
\begin{array}{l}
\PSL(2,r)\;(r=7,8,11,13,16,17,19,25,27,29,31,32,37,41,49,61), \\
\PSL(3,3),\,\PSL(3,5),\,\PSL(4,3),\,\PSL(4,5), \\
\PSU(3,3),\,\PSU(3,8),\,\PSU(4,2), \\
\PSp(4,5),\,\PSp(6,2),\,\Omega_8^+(2),\\
\text{G}_2(3),\,^3\!\text{D}_4(2),\,^2\!\text{F}_4(2)',\,^2\!\text{B}_2(8),\,^2\!\text{B}_2(32).
\end{array}
\]

Here $\PSL(2,37)$, $\PSL(4,3)$, $\PSL(4,5)$ and $\PSp(4,5)$ only come up in case $p=2$ and $^2\!\text{B}_2(32)$ only comes up for $p=5$.
Using that these subgroups only come up in cross characteristic, we obtain
\begin{align*}
N_{5e}(2) &\leq 57, \ N_{5e}(3) \leq 59, \ N_{5e}(5) \leq 64, \ N_{5e}(7) \leq 62, \\
\ N_{5e}(p) &\leq 67 \ (p \in \{11,13,17,19,29,31,41,61 \}), \ N_{5e}(p) \leq 69 \text{ else.} 
\end{align*}
So the maximal bound we obtain for $N_{5e}$ occurs for $p=2$ and is 119.

Summing up, we obtain $N_1 + N_2 + N_3 + N_4 + N_5 \leq 166$ and the proposition follows. 
\end{proof}

\begin{prop}\label{Exceptional}
If $G$ is a simple exceptional group of Lie type, then $\mathcal{J}(G)<2$. In particular, $G$ is HS.
\end{prop}
\begin{proof}
Let $G$ be a simple exceptional group of Lie type over a field of order $q=p^a$, for $p$ prime. Denote by $\ell$ the number of isomorphism types of maximal subgroups of $G$ and by $m_1$ the smallest index of a proper subgroup of $G$. Let $C$ be a number such that $\ell - \log_2\log_2 q \leq C$. We will use Corollary~\ref{cor:SimplificationSimpleGroups} with $\mathcal{B}_1$ in all cases with the data given in Table \ref{TableExceptionalGroups}. I.e. we verify that
\begin{equation} \label{ineqfg}
\mathcal{J}(G)-1 \leq (e^{\gamma}+1) \cdot (\ell/m_1) \cdot \ln\ln(|G|/m_1) < 1
\end{equation}
in all cases. 

We comment on the sources of the data in Table~\ref{TableExceptionalGroups}. 
The content of the $m_1$ column was computed in a series of papers by Vasil'ev and is compiled in \cite[Table 4]{GMPS}. We note that for the groups $E_7(q)$ this table contains a typo in a sign (cf. \cite{e6e7e8}).
The upper bound for $C$ follows from 
\cite[Theorem 4.1]{Wilson} for ${^2}\text{B}_2(q)$, 
\cite[Table 4.1]{Wilson} for $\text{G}_2(q)$, 
\cite[Theorem 4.2]{Wilson} for ${^2}\text{G}_2(q)$, 
\cite[Theorem 4.3]{Wilson} for ${^3}\text{D}_4(q)$ and 
\cite[Theorem 4.5]{Wilson} for ${^2}\text{F}_4(q)$. Moreover, for $\text{E}_8(q)$ we use Proposition \ref{E8lbound}.

For $\text{F}_4(q)$ we use \cite[Tables 1, 7, 8]{CravenF4E62E6}: note here that, taking into account the second and third columns, \cite[Table 1]{CravenF4E62E6} contributes $2$ for $p=2$ and at most $6$ for $p$ odd, \cite[Table 7]{CravenF4E62E6} contributes at most $13$ in case $p$ odd and \cite[Table 8]{CravenF4E62E6} contributes at most $17$ when $p=2$. So in total we obtain the value $19$. 

For $\text{E}_6(q)$ we use \cite[Tables 2, 9]{CravenF4E62E6}: here \cite[Table 6]{CravenF4E62E6} contributes at most $5$, taking into account the second and third column, and \cite[Table 9]{CravenF4E62E6} contributes at most 22. 

For ${^2}\text{E}_6(q)$ we use \cite[Tables 3, 10]{CravenF4E62E6}: here \cite[Table 3]{CravenF4E62E6} contributes at most $5$, taking into account the second and third columns, and \cite[Table 10]{CravenF4E62E6} contributes at most $20$, taking into account the second column.

For $\text{E}_7(q)$ we use \cite[Tables 1.1, 1.2, 4.1]{CravenE7}: here, taking into account the second and third columns, \cite[Table 1.1]{CravenE7} contributes $7$, for \cite[Table 1.2]{CravenE7} we take into account all the possible almost simple groups with the given socle which gives at most $9$ for the first line of the table and \cite[Table 4.1]{CravenE7} contributes $34$.

In order to prove \eqref{ineqfg}, we choose a number $r$ such that $|G| \leq q^r$. Then we bound
\begin{equation}\label{BoundsExceptional}
\mathcal{J}(G)-1 \leq (e^{\gamma}+1) \cdot (\ell/m_1) \cdot \ln \ln (|G|/m_1) \leq  \frac{2.8 \cdot (C+\log_2\log_2 q)}{m_1} \ln (r \ln q).
\end{equation}
The values of $C$ and $r$ are given in Table \ref{TableExceptionalGroups}. Since $m_1=m_1(q)$ is always a polynomial of degree at least $2$, the right-hand side of \eqref{BoundsExceptional} is a function of $q$ that is decreasing for $q \geq 2$, which we evaluate at the smallest available value of $q$ and deduce that $\mathcal{J}(G) < 2$ in all cases by \cite{GitRepMathematica}. 

\begin{table}[ht] 
\caption{Data about exceptional groups of Lie type}
\begin{tabular}{|c|c|c|c|c|} \hline
\mbox{Group} & $|G|$ & $C$ & $r$ & $m_1$ \\ \hline
${^2}\text{B}_2(q) = \mbox{Sz}(q)$, & $(q-1)q^2(q^2+1)$ & $4$ & $6$ & $q^2+1$ \\
$q=2^{2k+1} \geq 8$ & & & & \\ \hline
$\text{G}_2(q)$ & $q^6 (q^6-1) (q^2-1)$ & $11$ & $14$ & $(q^6-1)/(q-1)$ \mbox{except } \\
$q > 2$ & & & & $m_1(3)=351$,\ $m_1(4)=416$ \\ \hline
${^2}\text{G}_2(q)$ & $(q^3+1)q^3(q-1)$ & $5$ & $8$ & $q^3+1$ \\
$q=3^{2k+1} \geq 27$ & & & & \\ \hline
${^3}\text{D}_4(q)$ & $q^{12} (q^8+q^4+1) (q^6-1) (q^2-1)$ & $12$ & $29$ & $(q+1)(q^8+q^4+1)$ \\ \hline
$\text{F}_4(q)$ & $q^{24} \prod_{i \in I} (q^i-1)$ & $19$ & $52$ & $(q^4+1)(q^{12}-1)/(q-1)$ \\ 
& $I=\{2,6,8,12\}$ & & & \\ \hline
${^2}\text{F}_4(q)$ & $q^{12} (q^6+1) (q^4-1) (q^3+1) (q-1)$ & $9$ & $28$ & $(q^6+1)(q+1)(q^3+1)$ \\
$q=2^{2k+1} \geq 8$ & & & & \\ \hline
$\text{E}_6(q)$ & $\frac{1}{(3,q-1)} q^{36} \prod_{i \in I} (q^i-1)$ & $27$ & $78$ & $\frac{(q^9-1)(q^8+q^4+1)}{(q-1)}$ \\ 
& $I=\{2,5,6,8,9,12\}$ & & & \\ \hline
${^2}\text{E}_6(q)$ & $\frac{1}{(3,q+1)} q^{36} \prod_{i \in I} (q^i-(-1)^i)$ & $25$ & $82$ & $\frac{(q^{12}-1)(q^6-q^3+1)(q^4+1)}{(q-1)}$ \\
& $I=\{2,5,6,8,9,12\}$ & & & \\ \hline
$\text{E}_7(q)$ & $\frac{1}{(2,q-1)} q^{63} \prod_{i \in I} (q^i-1)$ & $50$ & $133$ & $\frac{(q^{14}-1)(q^9+1)(q^5+1)}{(q-1)}$ \\
& $I=\{2,6,8,10,12,14,18\}$ & & & \\ \hline
$\text{E}_8(q)$ & $q^{120} \prod_{i \in I} (q^i-1)$ & $166$ & $248$ &$\frac{(q^{30}-1)(q^{12}+1)(q^{10}+1)(q^6+1)}{(q-1)}$ \\
& $I=\{2,8,12,14,18,20,24,30\}$ & & & \\ \hline
\end{tabular} \label{TableExceptionalGroups}
\end{table}
The proof is completed.
\end{proof}

\section{Proof of main theorems}

The proofs of the main theorems are now direct consequences of the work done above and the CFSG. We collect the necessary information.

\textit{Proof of Theorem~\ref{th:Symmetric}}: Let $G$ be a symmetric group. If $G$ is finite and $n \geq 7$, the statement is a consequence of Proposition \ref{symfbound}. For $n \leq 6$ the result follows from \cite[Theorem A]{MS}. If $G$ is an infinite symmetric group, then it has no proper subgroup of finite index as a consequence of the Schreier-Ulam-Baer Theorem \cite[Satz]{Baer}. \hfill \qed

\textit{Proof of Theorem~\ref{th:Simple}}: For finite simple groups this follows from the Classification of Finite Simple Groups and Propositions \ref{Sporadics}, \ref{symfbound}, \ref{Classical} and \ref{Exceptional}. For infinite simple groups the claim is clear, as an infinite simple group has no proper subgroups of finite index. Indeed, the normal core of any subgroup of finite index also has finite index.  \hfill \qed

\textit{Proof of Theorem~\ref{th:Asymptotic}}: This is a consequence of the inequalities appearing in the proofs of the propositions above. More specifically, the inequalities given in \eqref{BoundJAn}, 
\eqref{BoundJClassicalqBigger2nBig}, 
\eqref{14case}, 
\eqref{BoundJClassicalqBigger2n13}, 
\eqref{BoundsnSmall} (together with the information in Tables~\ref{TableGm1} and \ref{TableC}) and \eqref{BoundsExceptional} (together with the information in Table~\ref{TableExceptionalGroups}).  \hfill \qed

\begin{rem}
One could ask how far the results we achieved above can be generalized to other classes of groups, such as almost simple groups or direct products of simple groups. Though for some of these subclasses, for example for extension of simple groups by groups of prime order, this is likely possible, there is no hope that the methods presented here will suffice to prove the Herzog-Sch\"onheim Conjecture for all such groups.

To see this define a function $\alpha: \mathbb{N} \rightarrow \mathbb{N}$ by $\alpha(n) = p_1 p_2 \cdots p_n$ where $p_1$,...,$p_n$ denote the smallest $n$ pairwise different primes. Then the simple group $G_n = \PSL(2, 2^{\alpha(n)})$ has an outer automorphism $\sigma$ of order $\alpha(n)$, given by the Galois action on the finite field $\mathbb{F}_{2^{\alpha(n)}}$. So the almost simple group $G_n \rtimes \langle \sigma \rangle$ contains maximal subgroups of index $p_1$, $p_2$,...,$p_n$. In particular, $\mathcal{J}(G) \geq 1+\sum_{i=1}^n \frac{1}{p_i}$ which is unbounded. Similar arguments of course apply to other groups of Lie type so that the function $\mathcal{J}$ takes values bigger than $2$ for many almost simple groups.

Neither are our methods sufficient to obtain a positive result for direct products of non-abelian simple groups: again denote by $p_1$, $p_2$,..,$p_n$ the smallest $n$ primes for any natural number $n$. Define the direct product of alternating groups $H_n = A_{p_1} \times A_{p_2} \times \cdots \times A_{p_n}$. The group $H_n$ also has maximal subgroups of of index $p_1$, $p_2$, ...,$p_n$. Hence, $\mathcal{J}(H_n)$ is again unbounded.
\end{rem}

\bibliographystyle{abbrv}
\bibliography{HS}

\end{document}